\documentclass[pdflatex,sn-mathphys-num]{sn-jnl}

\usepackage{amsmath,amssymb,amsthm,mathtools,bm}
\usepackage{mathrsfs}
\usepackage{hyperref}
\usepackage{enumitem}
\usepackage{microtype}

\hypersetup{
  colorlinks=true,
  linkcolor=blue,
  citecolor=blue,
  urlcolor=blue
}

\theoremstyle{plain}
\newtheorem{theorem}{Theorem}[section]
\newtheorem{corollary}[theorem]{Corollary}
\newtheorem{proposition}[theorem]{Proposition}
\newtheorem{lemma}[theorem]{Lemma}

\theoremstyle{definition}

\newtheorem{remark}[theorem]{Remark}
\newtheorem{example}[theorem]{Example}

\begin{document}

\title[Article Title]{Parameter-Space Heat Flow, Gaussian Density Ratios, and Sharp Hermite Truncation Rates}
\author[1,2,3]{\fnm{Jae Wan} \sur{Shim}}

\affil[1]{\orgdiv{Extreme Materials Research Center}, \orgname{Korea Institute of Science and Technology}, \orgaddress{\street{5 Hwarang-ro 14-gil, Seongbuk}, \city{Seoul}, \postcode{02792}, \country{Republic of Korea}}}

\affil[2]{\orgdiv{Climate and Environmental Research Institute}, \orgname{Korea Institute of Science and Technology}, \orgaddress{\street{5 Hwarang-ro 14-gil, Seongbuk}, \city{Seoul}, \postcode{02792}, \country{Republic of Korea}}}

\affil[3]{\orgdiv{Division of AI-Robotics, KIST Campus}, \orgname{University of Science and Technology}, \orgaddress{\street{5 Hwarang-ro 14-gil, Seongbuk}, \city{Seoul}, \postcode{02792}, \country{Republic of Korea}}}

\abstract{We reinterpret the classical Hermite generating function as a Gaussian density
ratio: relative to the unit Gaussian reference, it is the density ratio of a
Gaussian with shifted mean and unchanged covariance.  Applying the heat
semigroup in the mean-parameter variable to this generating function produces
the corresponding temperature variation. Thus the heat-semigroup time variable is reinterpreted
as the temperature variation of the Gaussian density ratio.

This parameter-space formulation also gives a parabolic control principle for
Hermite approximation errors.  Since Hermite projections act in the velocity
variable and the heat flow acts in the mean variable, Hermite block energies
and truncation tails are subsolutions of the same parameter-space heat
equation.  This remains useful for heat-evolved non-Gaussian perturbations
where no usable closed coefficient formula is available.

For Gaussian density ratios with general covariance, the Hermite coefficients
satisfy a weighted homogeneity in the mean and covariance-defect parameters.
This yields Ornstein--Uhlenbeck covariance, an exact generating function for
total-degree Hermite block energies, and the sharp geometric Hermite
truncation rate, equal to the square root of the largest absolute covariance
defect.  We also derive precise isotropic block and tail asymptotics and
interpret the estimates for near-Gaussian kinetic distributions.}

\maketitle

\section{Introduction}

Hermite expansions are a natural spectral representation for functions on
Gaussian space \cite{Janson1997,BakryGentilLedoux2014}. In kinetic theory,
they provide a natural way to represent distributions near a Maxwellian
equilibrium relative to a fixed reference Maxwellian
\cite{Grad1949Kinetic,Struchtrup2005,Levermore1996}. A basic question is
therefore how a Maxwellian is represented in Hermite variables with respect to
such a reference, and how fast its total-degree Hermite truncations converge.

For \(\theta>0\) and \(u\in\mathbb R^d\), let
\[
M_{\theta,u}(v)
:=
(2\pi\theta)^{-d/2}
\exp\!\left(
-\frac{|v-u|^2}{2\theta}
\right)
\]
be the Maxwellian with mean velocity \(u\) and temperature parameter
\(\theta\). We fix the unit-temperature Maxwellian
\[
w(v):=M_{1,0}(v)
\]
as the reference state.

The starting point of this paper is the function
\[
\Phi(u;v)
=
\exp\!\left(
u\cdot v-\frac{|u|^2}{2}
\right).
\]
In the usual Hermite theory, \(\Phi\) is the generating function for the
probabilists' Hermite polynomials \cite{DLMF,Grad1949Hermite}. In the present
kinetic interpretation, the same function has a second meaning: it is the
normalized Maxwellian corresponding to a pure mean shift at the reference
temperature. Indeed,
\[
\Phi(u;v)
=
\frac{M_{1,u}(v)}{w(v)}.
\]
Thus the Hermite generating function is already a Maxwellian object.

The key observation of this paper is that temperature variation can be
introduced at the level of this generating function by applying the heat
semigroup in the mean parameter \(u\).  For \(\tau\ge0\),
\[
\frac{M_{1+\tau,u}(v)}{w(v)}
=
\exp\!\left(
\frac{\tau}{2}\Delta_u
\right)
\Phi(u;v).
\]
Thus heat time in the mean-parameter space is identified with the temperature
increment relative to the reference state \(M_{1,0}\).  We write
\[
\Phi_\tau(u;v)
:=
\exp\!\left(
\frac{\tau}{2}\Delta_u
\right)\Phi(u;v)
=
\frac{M_{1+\tau,u}(v)}{w(v)}.
\]

Expanding \(\Phi_\tau\) in Hermite polynomials gives
\[
\frac{M_{1+\tau,u}(v)}{w(v)}
=
\sum_{\alpha\in\mathbb N^d}
\frac{1}{\alpha!}
A_\alpha(\tau,u)\mathrm{He}_\alpha(v),
\qquad
A_\alpha(\tau,u)
=
\exp\!\left(
\frac{\tau}{2}\Delta_u
\right)u^\alpha.
\]
Thus the temperature dependence of the Hermite coefficients is encoded by
heat evolution of monomials in the mean parameter.

We also use the parameter-space heat flow to control Hermite approximation
errors at the level of energies.  Since Hermite projections act in the
velocity variable and the heat semigroup acts in the mean variable, the two
operations commute.  Consequently, for any \(L^2(\gamma_d)\)-valued family
evolving by heat flow in the mean parameter, Hermite block energies and
truncation tails satisfy parabolic subsolution identities in parameter space.
This remains useful for heat-evolved non-Gaussian perturbations, such as
Gaussian ratios multiplied by a smooth cutoff or residual factor, where the
coefficient generating transform is generally not available in a usable closed
form.

The Gaussian part of the theory extends from isotropic heating to
full-covariance density ratios.  Let \(S=S^T\) with \(I+S>0\), and interpret
\(S\) as the covariance defect from the unit Gaussian reference.  Define
\[
K_{u,S}(v)
=
\frac{g_{u,I+S}(v)}{w(v)}.
\]
The Hermite coefficients are encoded by the generating function
\[
\sum_{\alpha\in\mathbb N^d}
\frac{z^\alpha}{\alpha!}A_\alpha(S,u)
=
\exp\!\left(
u\cdot z+\frac12 z^TSz
\right).
\]
This formula exhibits a weighted homogeneity:
\[
A_\alpha(\lambda^2S,\lambda u)
=
\lambda^{|\alpha|}A_\alpha(S,u).
\]
Thus the Hermite degree \(|\alpha|\) is identified with the weighted degree of
the Gaussian parameters, where \(u\) has weight \(1\) and \(S\) has weight
\(2\).

This weighted homogeneity is the organizing principle of the paper.  It implies
that the Ornstein--Uhlenbeck semigroup in the velocity variable corresponds,
within the Gaussian ratio family, to contraction of the Gaussian parameters.
We use the standard Gaussian-space normalization of the Ornstein--Uhlenbeck
semigroup \cite{Janson1997,BakryGentilLedoux2014}:
\[
T_tK_{u,S}
=
K_{e^{-t}u,e^{-2t}S}.
\]
It also gives an exact identity for the Hermite block energies.  If \(\Pi_m\)
denotes the projection onto the \(m\)-th total-degree Hermite chaos, then
\[
\sum_{m=0}^{\infty}
t^{2m}
\|\Pi_mK_{u,S}\|_{L^2(\gamma_d)}^2
=
\|K_{tu,t^2S}\|_{L^2(\gamma_d)}^2.
\]
Using the exact \(L^2\) norm of Gaussian density ratios, this becomes
\[
\sum_{m=0}^{\infty}
t^{2m}
\|\Pi_mK_{u,S}\|_{L^2(\gamma_d)}^2
=
\det(I-t^4S^2)^{-1/2}
\exp\!\left(
t^2u^T(I-t^2S)^{-1}u
\right).
\]

The singularities of this generating function determine the sharp Hermite
truncation rate, in the standard sense of coefficient asymptotics and singularity
analysis \cite{FlajoletSedgewick2009}.  In particular, if
\[
0<\|S\|_{\mathrm{op}}<1,
\]
then
\[
\limsup_{M\to\infty}
\|K_{u,S}-\Pi_{\le M}K_{u,S}\|_{L^2(\gamma_d)}^{1/M}
=
\|S\|_{\mathrm{op}}^{1/2}.
\]
Thus the sharp \(M\)-th root rate of fixed-reference Hermite approximation is
determined exactly by the square root of the operator norm of the covariance
defect.

In the isotropic heating case, the sharp rate becomes the scalar heating
parameter.  There the block-energy generating function reduces to a
Laguerre-type generating function, which yields precise block and tail
asymptotics through classical Plancherel--Rotach asymptotics
\cite{Szego1975,DLMF}.

The final part of the paper interprets these estimates for near-Gaussian
kinetic distributions, in the spirit of Hermite and moment approximations in
kinetic theory \cite{Grad1949Kinetic,Levermore1996,Struchtrup2005}.  If a
normalized kinetic density admits a decomposition
\[
g=\rho K_{u,S}+h,
\]
then the Hermite truncation error separates into a Gaussian-core contribution
and a non-Gaussian residual contribution.  When the residual is spectrally
negligible relative to the Gaussian core, the full approximation error has the
same sharp rate as the Gaussian core.  In the isotropic heating case, the
precise asymptotics give the leading-order Hermite spectral tail.

The paper is organized as follows.  The next section introduces the Hermite
notation, proves the parameter-space heat-flow representation of the normalized
Maxwellian ratio, and derives the parabolic control principle for Hermite block
energies and truncation tails.  We then record the isotropic prototype and the
Ornstein--Uhlenbeck covariance.  The following sections treat full-covariance
Gaussian density ratios, prove the exact \(L^2\) norm formula, and derive the
exact block-energy generating function.  We then prove the sharp Hermite root
rate and compare it with standard Hermite--Sobolev and analytic-vector
estimates.  The final sections give precise isotropic block and tail
asymptotics and apply the estimates to kinetic Hermite approximation.

\section{Setup and notation}

Let $d\ge 1$, and let $u,v\in\mathbb R^d$. For a multi-index
$\alpha=(\alpha_1,\dots,\alpha_d)\in\mathbb N^d$, write
\[
|\alpha|:=\alpha_1+\cdots+\alpha_d,
\qquad
\alpha!:=\alpha_1!\cdots \alpha_d!,
\qquad
u^\alpha:=u_1^{\alpha_1}\cdots u_d^{\alpha_d}.
\]
Let
\[
\mathrm{He}_\alpha(v):=\prod_{j=1}^d \mathrm{He}_{\alpha_j}(v_j),
\]
where $\mathrm{He}_{\alpha_j}$ denotes the one-dimensional probabilists' Hermite polynomial.

The probabilists' Hermite polynomials are characterized by the generating
function \cite{DLMF,Grad1949Hermite}
\[
\Phi(u;v):=\exp\!\left(u\cdot v-\frac{|u|^2}{2}\right)
=
\sum_{\alpha\in\mathbb N^d}
\frac{u^\alpha}{\alpha!}\,\mathrm{He}_\alpha(v).
\]

For \(\tau\ge0\), define the temperature-extended Hermite generating function
\[
\Phi_\tau(u;v)
:=
\exp\!\left(\frac{\tau}{2}\Delta_u\right)\Phi(u;v).
\]
For notational convenience, we also write
\[
U(\tau,u;v):=\Phi_\tau(u;v).
\]
Define
\[
A_\alpha(\tau,u)
:=
\exp\!\left(\frac{\tau}{2}\Delta_u\right)u^\alpha.
\]
Then
\[
\Phi_\tau(u;v)
=
U(\tau,u;v)
=
\sum_{\alpha\in\mathbb N^d}
\frac{1}{\alpha!}A_\alpha(\tau,u)\,\mathrm{He}_\alpha(v).
\]
For each \(M\ge0\), define the total-degree truncations
\[
\Phi_{\le M}(u;v)
:=
\sum_{|\alpha|\le M}
\frac{u^\alpha}{\alpha!}\,\mathrm{He}_\alpha(v),
\]
and
\[
U_{\le M}(\tau,u;v)
:=
\sum_{|\alpha|\le M}
\frac{1}{\alpha!}\,
A_\alpha(\tau,u)\,\mathrm{He}_\alpha(v).
\]
We work on the Gaussian space
\[
(\mathbb R^d,\gamma_d),
\qquad
\gamma_d(dv):=(2\pi)^{-d/2}e^{-|v|^2/2}\,dv.
\]
For $(u,\sigma)\in\mathbb R^d\times\mathbb R$, define
\[
K_{u,\sigma}(v):=U(\sigma^2,u;v).
\]

For each $m\ge0$, let
\[
\mathcal H_m:=\operatorname{span}\{\mathrm{He}_\alpha:\ |\alpha|=m\}
\subset L^2(\gamma_d),
\]
and let
\[
\Pi_m:L^2(\gamma_d)\to \mathcal H_m
\]
denote the orthogonal projection. We also write
\[
\Pi_{\le M}:=\sum_{m=0}^M \Pi_m.
\]
Since
\[
\left\{\frac{\mathrm{He}_\alpha}{\sqrt{\alpha!}}\right\}_{\alpha\in\mathbb N^d}
\]
is an orthonormal basis of $L^2(\gamma_d)$ \cite{Janson1997,BakryGentilLedoux2014},
the spaces $\mathcal H_m$ identify the total-degree Hermite decomposition in
the velocity variable $v$.

\subsection{Heat-flow identification of the normalized Gaussian density family}

For \(\theta>0\) and \(u\in\mathbb R^d\), define
\[
M_{\theta,u}(v)
:=
(2\pi\theta)^{-d/2}
\exp\!\left(-\frac{|v-u|^2}{2\theta}\right),
\qquad v\in\mathbb R^d.
\]
We use the unit-temperature reference
\[
w(v):=M_{1,0}(v)=(2\pi)^{-d/2}e^{-|v|^2/2}.
\]
For \(\tau\ge0\), the normalized Maxwellian ratio at temperature
\(\theta=1+\tau\) is
\[
\frac{M_{1+\tau,u}(v)}{w(v)}.
\]

\begin{proposition}[Heat-kernel representation and closed form]\label{prop:explicit_heat_evolution_Hermite_generator}
For the temperature-extended generating function
\(U(\tau,u;v)=\Phi_\tau(u;v)\), the following heat-kernel representation holds.
For \(\tau>0\),
\begin{equation}
U(\tau,u;v)
=
\int_{\mathbb R^d}
\frac{1}{(2\pi\tau)^{d/2}}
\exp\!\left(-\frac{|u-z|^2}{2\tau}\right)
\Phi(z;v)\,dz.
\label{eq:U_heat_kernel_representation}
\end{equation}
Moreover, this integral is explicitly equal to
\begin{equation}
U(\tau,u;v)
=
(1+\tau)^{-d/2}
\exp\!\left(
\frac{\tau |v|^2+2u\cdot v-|u|^2}{2(1+\tau)}
\right).
\label{eq:U_closed_form}
\end{equation}
In particular,
\[
U(0,u;v)=\Phi(u;v)
\]
by taking \(\tau\downarrow0\).
\end{proposition}

\begin{proof}
The heat-kernel representation follows directly from the standard heat
semigroup formula in the \(u\)-variable \cite{BakryGentilLedoux2014}:
\[
e^{\frac{\tau}{2}\Delta_u}f(u)
=
\int_{\mathbb R^d}
(2\pi\tau)^{-d/2}
\exp\!\left(-\frac{|u-z|^2}{2\tau}\right)f(z)\,dz,
\qquad \tau>0.
\]
Applying this to \(f(z)=\Phi(z;v)\) gives
\eqref{eq:U_heat_kernel_representation}.

It remains to compute the Gaussian integral. If \(Z\sim N(u,\tau I)\), then
\[
U(\tau,u;v)
=
\mathbb E\exp\!\left(Z\cdot v-\frac{|Z|^2}{2}\right).
\]
Writing \(Z=u+\sqrt{\tau}\,Y\), \(Y\sim N(0,I)\), we get
\[
Z\cdot v-\frac{|Z|^2}{2}
=
u\cdot v-\frac{|u|^2}{2}
+
\sqrt{\tau}\,Y\cdot(v-u)
-
\frac{\tau}{2}|Y|^2.
\]
Hence
\[
U(\tau,u;v)
=
\exp\!\left(u\cdot v-\frac{|u|^2}{2}\right)
\mathbb E
\exp\!\left(
\sqrt{\tau}\,Y\cdot(v-u)
-\frac{\tau}{2}|Y|^2
\right).
\]
Using the elementary Gaussian identity
\[
\mathbb E
\exp\!\left(a\cdot Y-\frac{\tau}{2}|Y|^2\right)
=
(1+\tau)^{-d/2}
\exp\!\left(\frac{|a|^2}{2(1+\tau)}\right),
\qquad Y\sim N(0,I),
\]
with \(a=\sqrt{\tau}(v-u)\), we obtain
\[
U(\tau,u;v)
=
(1+\tau)^{-d/2}
\exp\!\left(
u\cdot v-\frac{|u|^2}{2}
+
\frac{\tau|v-u|^2}{2(1+\tau)}
\right).
\]
Finally,
\[
u\cdot v-\frac{|u|^2}{2}
+
\frac{\tau|v-u|^2}{2(1+\tau)}
=
\frac{\tau |v|^2+2u\cdot v-|u|^2}{2(1+\tau)}.
\]
This proves \eqref{eq:U_closed_form}.
\end{proof}

\begin{remark}[Normalized Maxwellian ratio]
The proposition identifies
\[
U(\tau,u;\cdot)=\frac{M_{1+\tau,u}}{w}.
\]
Hence \(A_\alpha(\tau,u)\) are precisely the Hermite coefficients of the
normalized Maxwellian ratio.  The \(L^2(\gamma_d)\) integrability threshold
and the exact norm are computed below in the full-covariance setting.
\end{remark}

\subsection{Parabolic control of Hermite block energies and tails}
\label{subsec:parabolic_control_hermite_tails}

We next isolate the part of the heat-flow argument that does not depend on the
explicit Gaussian coefficient formula.  The result applies to any
\(L^2(\gamma_d)\)-valued heat flow in the mean-parameter variable and gives
parabolic subsolution identities for Hermite block energies and truncation
tails.  Let
\[
F_\tau(u;\cdot)\in L^2(\gamma_d),
\qquad \tau\ge0,\quad u\in\mathbb R^d,
\]
be sufficiently smooth in \((\tau,u)\) and suppose that
\[
\partial_\tau F_\tau
=
\frac12\Delta_uF_\tau
\]
as an \(L^2(\gamma_d)\)-valued heat equation.  For each \(m\ge0\), define
\[
E_m(\tau,u)
:=
\|\Pi_mF_\tau(u;\cdot)\|_{L^2(\gamma_d)}^2,
\]
and for each \(M\ge0\), define the truncation-tail energy
\[
T_M(\tau,u)
:=
\|(I-\Pi_{\le M})F_\tau(u;\cdot)\|_{L^2(\gamma_d)}^2.
\]

\begin{theorem}[Parabolic subsolution identities for Hermite energies]
\label{thm:parabolic_control_hermite_block_energies}
Let \(F_\tau\) solve
\[
\partial_\tau F_\tau
=
\frac12\Delta_uF_\tau
\]
in the sense above.  Then each Hermite block energy satisfies
\begin{equation}
\left(
\partial_\tau-\frac12\Delta_u
\right)
E_m(\tau,u)
=
-
\|\nabla_u\Pi_mF_\tau(u;\cdot)\|_{L^2(\gamma_d;\mathbb R^d)}^2
\le0.
\label{eq:block_energy_subsolution}
\end{equation}
Similarly, each Hermite truncation-tail energy satisfies
\begin{equation}
\left(
\partial_\tau-\frac12\Delta_u
\right)
T_M(\tau,u)
=
-
\|\nabla_u(I-\Pi_{\le M})F_\tau(u;\cdot)\|_{L^2(\gamma_d;\mathbb R^d)}^2
\le0.
\label{eq:tail_energy_subsolution}
\end{equation}
If, in addition,
\[
F_\tau=e^{(\tau/2)\Delta_u}F_0
\]
and the heat-kernel representation is justified, then
\[
E_m(\tau,u)
\le
e^{(\tau/2)\Delta_u}E_m(0,u),
\qquad
T_M(\tau,u)
\le
e^{(\tau/2)\Delta_u}T_M(0,u).
\]
\end{theorem}

\begin{proof}
The Hermite projection \(\Pi_m\) acts only on the velocity variable \(v\),
whereas \(\Delta_u\) acts only on the mean parameter \(u\).  Hence these
operators commute, and
\[
\partial_\tau\Pi_mF_\tau
=
\frac12\Delta_u\Pi_mF_\tau.
\]
Therefore
\[
\partial_\tau E_m
=
2\left\langle
\Pi_mF_\tau,\partial_\tau\Pi_mF_\tau
\right\rangle_{L^2(\gamma_d)}
=
\left\langle
\Pi_mF_\tau,\Delta_u\Pi_mF_\tau
\right\rangle_{L^2(\gamma_d)}.
\]
On the other hand,
\[
\frac12\Delta_uE_m
=
\left\langle
\Pi_mF_\tau,\Delta_u\Pi_mF_\tau
\right\rangle_{L^2(\gamma_d)}
+
\|\nabla_u\Pi_mF_\tau\|_{L^2(\gamma_d;\mathbb R^d)}^2.
\]
Subtracting the two identities gives
\[
\left(
\partial_\tau-\frac12\Delta_u
\right)
E_m
=
-
\|\nabla_u\Pi_mF_\tau\|_{L^2(\gamma_d;\mathbb R^d)}^2.
\]
This proves \eqref{eq:block_energy_subsolution}.  The proof of
\eqref{eq:tail_energy_subsolution} is identical, replacing \(\Pi_m\) by
\(I-\Pi_{\le M}\).  If \(F_\tau=e^{(\tau/2)\Delta_u}F_0\), the comparison
estimates above follow from the parabolic comparison principle under the usual
growth assumptions.  Equivalently, using the heat-kernel representation and
Jensen's inequality gives the same bounds directly.
\end{proof}

\begin{remark}[Why this goes beyond the closed Gaussian calculation]
For the Gaussian ratio \(F_\tau(u;v)=U(\tau,u;v)\), the coefficients and block
energies can be computed explicitly.  The preceding theorem is different in
nature: it uses only the \(L^2(\gamma_d)\)-valued heat equation in the
parameter variable.  Thus the block and tail estimates remain meaningful for heat-evolved families
for which the coefficient generating transform is not available in a usable
closed form.
\end{remark}

\begin{example}[A cutoff heat-flow family]
\label{ex:smooth_cutoff_heat_flow}
Let \(\chi\in C_c^\infty(\mathbb R^d)\), and define
\[
F_\tau^\chi(u;v)
:=
\chi(v)U(\tau,u;v)
=
\chi(v)\frac{M_{1+\tau,u}(v)}{w(v)}.
\]
Since \(\chi\) is independent of \(u\), this family still satisfies
\[
\partial_\tau F_\tau^\chi
=
\frac12\Delta_uF_\tau^\chi.
\]
Therefore the Hermite block energies
\[
E_m^\chi(\tau,u)
:=
\|\Pi_mF_\tau^\chi(u;\cdot)\|_{L^2(\gamma_d)}^2
\]
and the truncation tails
\[
T_M^\chi(\tau,u)
:=
\|(I-\Pi_{\le M})F_\tau^\chi(u;\cdot)\|_{L^2(\gamma_d)}^2
\]
satisfy the parabolic subsolution identities
\[
\left(
\partial_\tau-\frac12\Delta_u
\right)
E_m^\chi(\tau,u)
\le0,
\qquad
\left(
\partial_\tau-\frac12\Delta_u
\right)
T_M^\chi(\tau,u)
\le0.
\]

On the other hand, the coefficient generating transform is no longer a simple
Gaussian exponential.  Indeed, for real \(z\),
\[
\begin{aligned}
G_\chi(\tau,u;z)
&:=
\int_{\mathbb R^d}
e^{z\cdot v-(z\cdot z)/2}
F_\tau^\chi(u;v)\,d\gamma_d(v) \\
&=
e^{u\cdot z+\tau(z\cdot z)/2}
\int_{\mathbb R^d}
\chi(y)\,
M_{1+\tau,\,u+(1+\tau)z}(y)\,dy .
\end{aligned}
\]
For a general cutoff \(\chi\), the last factor is the heat transform of
\(\chi\) evaluated at a shifted mean, and it need not reduce to an elementary
Gaussian exponential or determinant expression.  As a coefficient transform,
the identity may then be read formally or by analytic continuation where
appropriate.  Thus the coefficient-generating
route generally loses the explicit closed form available in the pure Gaussian
case, while the parameter-space heat-flow identities for block energies and
truncation tails remain exact.
\end{example}

\section{Isotropic prototype: Hermite blocks}
The isotropic case serves as a prototype for the full-covariance theory.  Here
the covariance defect is \(S=\sigma^2I\), and the corresponding density ratio is
\[
K_{u,\sigma}=U(\sigma^2,u;\cdot).
\]
We record the associated Hermite blocks and total-degree truncations; the
Ornstein--Uhlenbeck covariance and sharp truncation rate will follow from the
full-covariance theory below.

Assume \(|\sigma|<1\), so that \(K_{u,\sigma}\in L^2(\gamma_d)\). Since
\[
K_{u,\sigma}
=
\sum_{\alpha\in\mathbb N^d}
\frac{1}{\alpha!}
A_\alpha(\sigma^2,u)\mathrm{He}_\alpha
\qquad\text{in }L^2(\gamma_d),
\]
the \(m\)-th Hermite block is
\[
\Pi_mK_{u,\sigma}
=
\sum_{|\alpha|=m}
\frac{1}{\alpha!}
A_\alpha(\sigma^2,u)\mathrm{He}_\alpha.
\]
Consequently, for \(M\ge0\),
\[
\Pi_{\le M}K_{u,\sigma}
=
\sum_{|\alpha|\le M}
\frac{1}{\alpha!}
A_\alpha(\sigma^2,u)\mathrm{He}_\alpha
=
U_{\le M}(\sigma^2,u;\cdot).
\]

\section{Gaussian density ratios with full covariance}

Throughout this section, \(d\gamma_d=w\,dv\) is the unit Gaussian reference
measure fixed above.  We use standard notation for Gaussian measures and
Gaussian density ratios \cite{Bogachev1998,Janson1997}.

For \(u\in\mathbb R^d\) and a real symmetric matrix \(S\in\mathbb R^{d\times d}\)
such that
\[
I+S>0,
\]
we define
\[
K_{u,S}(v)
:=
\frac{g_{u,I+S}(v)}{w(v)},
\]
where
\[
g_{u,\Sigma}(v)
:=
(2\pi)^{-d/2}\det(\Sigma)^{-1/2}
\exp\!\left(
-\frac12(v-u)^T\Sigma^{-1}(v-u)
\right)
\]
denotes the Gaussian density with mean \(u\) and covariance matrix
\(\Sigma>0\). Equivalently,
\[
K_{u,S}(v)
=
\det(I+S)^{-1/2}
\exp\!\left(
-\frac12(v-u)^T(I+S)^{-1}(v-u)
+\frac12|v|^2
\right).
\]

The isotropic shorthand is recovered by taking \(S=\sigma^2I\):
\[
K_{u,\sigma}
=
K_{u,\sigma^2I}
=
\frac{M_{1+\sigma^2,u}}{w}
=
U(\sigma^2,u;\cdot).
\]
Thus \(K_{u,\sigma}\) denotes the scalar isotropic family, while \(K_{u,S}\)
denotes the full-covariance family.

\subsection{Coefficient notation}

For the block-energy calculation below, we record the Hermite coefficients of
the full-covariance Gaussian ratio.  Define \(A_\alpha(S,u)\) by
\[
\sum_{\alpha\in\mathbb N^d}
\frac{z^\alpha}{\alpha!}A_\alpha(S,u)
=
\exp\!\left(
u\cdot z+\frac12z^TSz
\right).
\]
If \(K_{u,S}\in L^2(\gamma_d)\), then
\[
\langle K_{u,S},\mathrm{He}_\alpha\rangle_{L^2(\gamma_d)}
=
A_\alpha(S,u),
\]
and therefore
\[
K_{u,S}
=
\sum_{\alpha\in\mathbb N^d}
\frac{1}{\alpha!}
A_\alpha(S,u)\mathrm{He}_\alpha
\qquad\text{in }L^2(\gamma_d).
\]
Indeed, if \(X\sim N(u,I+S)\), then
\[
K_{u,S}(v)\,d\gamma_d(v)=g_{u,I+S}(v)\,dv,
\]
so
\[
\langle K_{u,S},\mathrm{He}_\alpha\rangle_{L^2(\gamma_d)}
=
\mathbb E[\mathrm{He}_\alpha(X)].
\]
Taking expectations in the Hermite generating function gives
\[
\sum_{\alpha\in\mathbb N^d}
\frac{z^\alpha}{\alpha!}
\langle K_{u,S},\mathrm{He}_\alpha\rangle_{L^2(\gamma_d)}
=
\mathbb E\exp\!\left(z\cdot X-\frac{|z|^2}{2}\right)
=
\exp\!\left(
u\cdot z+\frac12z^TSz
\right),
\]
and comparing coefficients gives the claim.

The same generating function gives the weighted homogeneity
\begin{equation}
A_\alpha(\lambda^2S,\lambda u)
=
\lambda^{|\alpha|}
A_\alpha(S,u),
\qquad \lambda\in\mathbb R.
\label{eq:anisotropic_weighted_homogeneity}
\end{equation}
Thus the Hermite degree \(|\alpha|\) agrees with the weighted homogeneous
degree in the Gaussian parameters \((u,S)\), where \(u\) has weight \(1\)
and \(S\) has weight \(2\).  This weighted homogeneity is the link between
Hermite spectral damping and Gaussian parameter contraction under the
Ornstein--Uhlenbeck semigroup.

\begin{remark}[Heat-flow interpretation and matrix covariance defects]
When \(S\ge0\), the coefficient generating function above is equivalently
obtained by applying the anisotropic parameter-space heat semigroup
\[
\exp\!\left(\frac12 S:\nabla_u^2\right)
\]
to the pure-shift generating function \(\Phi(u;v)\).  In this case \(S\)
represents a genuine covariance increment generated by a heat flow in the
mean-parameter variable.

For a general symmetric matrix \(S\) satisfying \(I+S>0\), however, \(S\) may
have negative eigenvalues.  In that case the same coefficient formula should
be understood as an algebraic Gaussian identity, or equivalently as the
analytic continuation of the positive-semidefinite heat-flow formula, rather
than as a genuine heat semigroup.
\end{remark}

\section{Ornstein--Uhlenbeck covariance}

Let
\[
L:=\Delta_v-v\cdot\nabla_v
\]
be the Ornstein--Uhlenbeck generator associated with the reference Gaussian
measure \(\gamma_d\). Equivalently, \(L\) is the infinitesimal generator of
the diffusion
\[
dX_t=-X_t\,dt+\sqrt2\,dB_t,
\]
whose invariant probability measure is \(\gamma_d\).

Set
\[
N:=-L.
\]
Then
\[
N\mathrm{He}_\alpha=|\alpha|\mathrm{He}_\alpha.
\]
The Ornstein--Uhlenbeck semigroup is
\[
T_t=e^{tL}=e^{-tN},
\qquad t\ge0.
\]
Hence, if
\[
f=\sum_{m=0}^\infty \Pi_m f
\]
is the Hermite chaos decomposition of \(f\in L^2(\gamma_d)\), then
\begin{equation}
T_tf
=
\sum_{m=0}^\infty e^{-mt}\Pi_m f.
\label{eq:OU_Hermite_spectral_form}
\end{equation}

\begin{theorem}[OU covariance of Gaussian density ratios]
\label{thm:anisotropic_OU_covariance}
Let \(u\in\mathbb R^d\) and let \(S=S^T\) satisfy \(I+S>0\).
Interpreting \(T_t\) as the Ornstein--Uhlenbeck Markov semigroup acting on
density ratios with respect to \(\gamma_d\), one has
\begin{equation}
T_tK_{u,S}
=
K_{e^{-t}u,e^{-2t}S},
\qquad t\ge0,
\label{eq:anisotropic_OU_covariance}
\end{equation}
in \(L^1(\gamma_d)\). If in addition \(\|S\|_{\mathrm{op}}<1\), then the
identity also holds in \(L^2(\gamma_d)\) and agrees with the Hermite spectral
action of \(T_t=e^{-tN}\).

In particular, in the isotropic heating parametrization \(S=\sigma^2I\),
\[
T_tK_{u,\sigma}
=
K_{e^{-t}u,e^{-t}\sigma}.
\]
\end{theorem}

\begin{proof}
The Ornstein--Uhlenbeck process has the explicit representation
\[
X_t=e^{-t}X_0+\sqrt{1-e^{-2t}}\,Z,
\qquad Z\sim N(0,I),
\]
where \(Z\) is independent of \(X_0\). If
\[
X_0\sim N(u,I+S),
\]
then
\[
X_t\sim N(e^{-t}u,I+e^{-2t}S).
\]

Since \(\gamma_d=N(0,I)\) is invariant for the Ornstein--Uhlenbeck semigroup
and the semigroup is reversible with respect to \(\gamma_d\), the forward
evolution of densities, when written as density ratios with respect to
\(\gamma_d\), is represented by the same operator \(T_t\). Thus the Gaussian
density ratio \(K_{u,S}\) evolves to the Gaussian density ratio corresponding
to \(N(e^{-t}u,I+e^{-2t}S)\). Hence
\[
T_tK_{u,S}
=
K_{e^{-t}u,e^{-2t}S}
\]
in \(L^1(\gamma_d)\).

If, in addition, \(\|S\|_{\mathrm{op}}<1\), then
\[
K_{u,S}\in L^2(\gamma_d),
\qquad
K_{e^{-t}u,e^{-2t}S}\in L^2(\gamma_d).
\]
In this \(L^2\)-regime, the same identity can be checked on Hermite
coefficients. Indeed,
\[
K_{u,S}
=
\sum_{\alpha\in\mathbb N^d}
\frac{1}{\alpha!}A_\alpha(S,u)\mathrm{He}_\alpha
\qquad\text{in }L^2(\gamma_d).
\]
Using
\[
T_t\mathrm{He}_\alpha=e^{-t|\alpha|}\mathrm{He}_\alpha,
\]
we obtain
\[
T_tK_{u,S}
=
\sum_{\alpha\in\mathbb N^d}
\frac{e^{-t|\alpha|}}{\alpha!}
A_\alpha(S,u)\mathrm{He}_\alpha.
\]
By the weighted homogeneity
\[
A_\alpha(e^{-2t}S,e^{-t}u)
=
e^{-t|\alpha|}A_\alpha(S,u),
\]
this becomes
\[
T_tK_{u,S}
=
\sum_{\alpha\in\mathbb N^d}
\frac{1}{\alpha!}
A_\alpha(e^{-2t}S,e^{-t}u)\mathrm{He}_\alpha
=
K_{e^{-t}u,e^{-2t}S}
\qquad\text{in }L^2(\gamma_d).
\]
\end{proof}

\section{Exact norm formula and critical boundary}

\begin{theorem}[Exact \(L^2\) norm and critical boundary]
\label{thm:anisotropic_exact_L2_norm}
Assume \(S=S^T\) and \(I+S>0\). Then
\[
K_{u,S}\in L^2(\gamma_d)
\quad\Longleftrightarrow\quad
\|S\|_{\mathrm{op}}<1.
\]
Moreover, if \(\|S\|_{\mathrm{op}}<1\), then
\begin{equation}
\|K_{u,S}\|_{L^2(\gamma_d)}^2
=
\det(I-S^2)^{-1/2}
\exp\!\left(
u^T(I-S)^{-1}u
\right).
\label{eq:anisotropic_exact_L2_norm}
\end{equation}
\end{theorem}

\begin{proof}
By definition,
\[
K_{u,S}(v)^2\,d\gamma_d(v)
=
(2\pi)^{-d/2}\det(I+S)^{-1}
\exp\!\left(
-(v-u)^T(I+S)^{-1}(v-u)
+\frac12|v|^2
\right)\,dv.
\]
The quadratic part is integrable precisely when
\[
2(I+S)^{-1}-I>0,
\]
which is equivalent to
\[
S<I.
\]
Together with \(I+S>0\), this is equivalent to
\[
\|S\|_{\mathrm{op}}<1.
\]
Completing the square gives
\[
\|K_{u,S}\|_{L^2(\gamma_d)}^2
=
\det(I-S^2)^{-1/2}
\exp\!\left(
u^T(I-S)^{-1}u
\right).
\]
\end{proof}

In the isotropic case \(S=\sigma^2I\), this becomes
\begin{equation}
\|K_{u,\sigma}\|_{L^2(\gamma_d)}^2
=
(1-\sigma^4)^{-d/2}
\exp\!\left(
\frac{|u|^2}{1-\sigma^2}
\right),
\qquad |\sigma|<1.
\label{eq:isotropic_exact_L2_norm}
\end{equation}

\section{Exact block-energy generating function}

Let \(\Pi_m\) denote the orthogonal projection in \(L^2(\gamma_d)\) onto the
\(m\)-th Hermite chaos
\[
\mathcal H_m
=
\operatorname{span}\{\mathrm{He}_\alpha:|\alpha|=m\}.
\]

\begin{theorem}[Anisotropic block-energy generating function]
\label{thm:anisotropic_block_energy_generating_function}
Assume \(S=S^T\), \(I+S>0\), and \(\|S\|_{\mathrm{op}}<1\). Then, for every
\(t\ge0\) such that
\[
t^2\|S\|_{\mathrm{op}}<1,
\]
we have
\begin{equation}
\sum_{m=0}^\infty
t^{2m}
\|\Pi_mK_{u,S}\|_{L^2(\gamma_d)}^2
=
\|K_{tu,t^2S}\|_{L^2(\gamma_d)}^2.
\label{eq:anisotropic_block_energy_abstract}
\end{equation}
Equivalently,
\begin{equation}
\sum_{m=0}^\infty
t^{2m}
\|\Pi_mK_{u,S}\|_{L^2(\gamma_d)}^2
=
\det(I-t^4S^2)^{-1/2}
\exp\!\left(
t^2u^T(I-t^2S)^{-1}u
\right).
\label{eq:anisotropic_block_energy_closed_form}
\end{equation}
\end{theorem}

\begin{proof}
By weighted homogeneity,
\[
\Pi_mK_{tu,t^2S}
=
t^m\Pi_mK_{u,S}.
\]
Since the Hermite chaos spaces are mutually orthogonal in \(L^2(\gamma_d)\),
Pythagoras' theorem gives
\[
\|K_{tu,t^2S}\|_{L^2(\gamma_d)}^2
=
\sum_{m=0}^\infty
\|\Pi_mK_{tu,t^2S}\|_{L^2(\gamma_d)}^2
=
\sum_{m=0}^\infty
t^{2m}
\|\Pi_mK_{u,S}\|_{L^2(\gamma_d)}^2.
\]
The closed form follows by applying
Theorem~\ref{thm:anisotropic_exact_L2_norm} to \(K_{tu,t^2S}\).
\end{proof}

In the isotropic case \(S=\sigma^2I\), this gives, for \(0<|\sigma|<1\),
\begin{equation}
\sum_{m=0}^\infty
t^{2m}
\|\Pi_mK_{u,\sigma}\|_{L^2(\gamma_d)}^2
=
(1-t^4\sigma^4)^{-d/2}
\exp\!\left(
\frac{t^2|u|^2}{1-t^2\sigma^2}
\right),
\qquad
0\le t<|\sigma|^{-1}.
\label{eq:isotropic_block_energy_closed_form}
\end{equation}
If \(\sigma=0\), the same identity holds for every \(t\ge0\), with right-hand
side equal to \(\exp(t^2|u|^2)\).

\section{Sharp Hermite root rate}

Define
\[
a_m(u,S)
:=
\|\Pi_mK_{u,S}\|_{L^2(\gamma_d)}^2.
\]

\begin{theorem}[Sharp anisotropic Hermite root rate]
\label{thm:anisotropic_sharp_root_rate}
Assume \(S=S^T\), \(I+S>0\), and
\[
0<\|S\|_{\mathrm{op}}<1.
\]
Then
\begin{equation}
\limsup_{m\to\infty}
a_m(u,S)^{1/m}
=
\|S\|_{\mathrm{op}}.
\label{eq:anisotropic_block_energy_root_rate}
\end{equation}
Equivalently,
\begin{equation}
\limsup_{m\to\infty}
\|\Pi_mK_{u,S}\|_{L^2(\gamma_d)}^{1/m}
=
\|S\|_{\mathrm{op}}^{1/2}.
\label{eq:anisotropic_block_norm_root_rate}
\end{equation}

Let
\[
R_M^{\mathrm{sp}}(u,S)
:=
K_{u,S}-\Pi_{\le M}K_{u,S}
=
\sum_{m=M+1}^\infty \Pi_mK_{u,S}.
\]
Then
\begin{equation}
\limsup_{M\to\infty}
\|R_M^{\mathrm{sp}}(u,S)\|_{L^2(\gamma_d)}^{1/M}
=
\|S\|_{\mathrm{op}}^{1/2}.
\label{eq:anisotropic_tail_root_rate}
\end{equation}
In particular, for every
\[
\|S\|_{\mathrm{op}}^{1/2}<\rho<1,
\]
there exists \(C=C(u,S,\rho)>0\) such that
\begin{equation}
\|R_M^{\mathrm{sp}}(u,S)\|_{L^2(\gamma_d)}
\le
C\rho^M,
\qquad M\ge0.
\label{eq:anisotropic_rho_tail_bound}
\end{equation}
The rate \(\|S\|_{\mathrm{op}}^{1/2}\) is optimal in the root-test sense.
\end{theorem}

\begin{proof}
Set
\[
F(z):=
\sum_{m=0}^\infty a_m(u,S)z^m.
\]
For real \(z=t^2\ge0\) sufficiently small, the block-energy identity gives
the closed form below. Since both sides are analytic in a neighborhood of
\(z=0\), the identity extends to complex \(z\) by analytic continuation.
By Theorem~\ref{thm:anisotropic_block_energy_generating_function},
\[
F(z)
=
\det(I-z^2S^2)^{-1/2}
\exp\!\left(
zu^T(I-zS)^{-1}u
\right).
\]
Diagonalize
\[
S=Q\Lambda Q^T,
\qquad
\Lambda=\operatorname{diag}(\lambda_1,\dots,\lambda_d),
\]
where \(Q\) is orthogonal. Writing
\[
\tilde u:=Q^Tu,
\]
we have
\[
\det(I-z^2S^2)^{-1/2}
=
\prod_{j=1}^d(1-z^2\lambda_j^2)^{-1/2},
\]
where the branch is chosen to be \(1\) at \(z=0\). Moreover,
\[
zu^T(I-zS)^{-1}u
=
z\sum_{j=1}^d
\frac{\tilde u_j^2}{1-z\lambda_j}.
\]
Therefore both the determinant factor and the exponential factor are analytic
in the disk
\[
|z|<\|S\|_{\mathrm{op}}^{-1}.
\]
Consequently \(F\) is analytic in this disk.

Let \(\lambda_j\) be an eigenvalue such that
\[
|\lambda_j|=\|S\|_{\mathrm{op}}.
\]
Then the determinant factor has a non-removable algebraic singularity at a
point of the circle
\[
|z|=\|S\|_{\mathrm{op}}^{-1}.
\]
At such a point, if the exponential factor is regular, then it is nonzero and
cannot cancel the algebraic singularity. If the exponential factor is singular,
then the product is singular a fortiori. Hence \(F\) has a non-removable
singularity on the circle
\[
|z|=\|S\|_{\mathrm{op}}^{-1}.
\]
Therefore the radius of convergence of \(F\) is exactly
\[
R=\|S\|_{\mathrm{op}}^{-1}.
\]
By the Cauchy--Hadamard formula,
\[
\limsup_{m\to\infty}a_m(u,S)^{1/m}
=
R^{-1}
=
\|S\|_{\mathrm{op}}.
\]
Taking square roots gives
\eqref{eq:anisotropic_block_norm_root_rate}.

Moreover,
\[
\|R_M^{\mathrm{sp}}(u,S)\|_{L^2(\gamma_d)}^2
=
\sum_{m=M+1}^\infty a_m(u,S).
\]
The upper bound follows from the root estimate: for every
\(\rho\) satisfying
\[
\|S\|_{\mathrm{op}}^{1/2}<\rho<1,
\]
there exists \(C_1>0\) such that
\[
a_m(u,S)\le C_1\rho^{2m}
\qquad(m\ge0).
\]
Hence
\[
\|R_M^{\mathrm{sp}}(u,S)\|_{L^2(\gamma_d)}^2
=
\sum_{m=M+1}^\infty a_m(u,S)
\le
C_1\sum_{m=M+1}^\infty \rho^{2m}
\le
C_2\rho^{2M}.
\]
Taking square roots gives the upper bound.

For the lower bound, fix \(\varepsilon>0\). Since
\[
\limsup_{m\to\infty}a_m(u,S)^{1/m}
=
\|S\|_{\mathrm{op}},
\]
there are infinitely many \(m\) such that
\[
a_m(u,S)\ge
\left(\|S\|_{\mathrm{op}}-\varepsilon\right)^m.
\]
For such \(m\),
\[
\|R_{m-1}^{\mathrm{sp}}(u,S)\|_{L^2(\gamma_d)}^2
=
\sum_{k=m}^\infty a_k(u,S)
\ge
a_m(u,S).
\]
Therefore
\[
\limsup_{M\to\infty}
\|R_M^{\mathrm{sp}}(u,S)\|_{L^2(\gamma_d)}^{1/M}
\ge
\left(\|S\|_{\mathrm{op}}-\varepsilon\right)^{1/2}.
\]
Letting \(\varepsilon\downarrow0\) gives the matching lower bound. This proves
\eqref{eq:anisotropic_tail_root_rate}.
\end{proof}

In the isotropic case \(S=\sigma^2I\), the sharp rate becomes
\[
\|S\|_{\mathrm{op}}^{1/2}=|\sigma|.
\]
Thus
\[
\limsup_{M\to\infty}
\|K_{u,\sigma}-\Pi_{\le M}K_{u,\sigma}\|_{L^2(\gamma_d)}^{1/M}
=
|\sigma|.
\]

\subsection{Comparison with the standard Hermite spectral bound}

Throughout this subsection, assume that \(S=S^T\) and
\[
0<\|S\|_{\mathrm{op}}<1.
\]
Then \(I+S>0\), and hence \(K_{u,S}\) is well-defined and belongs to
\(L^2(\gamma_d)\).

Recall that
\[
L=\Delta_v-v\cdot\nabla_v
\]
is the Ornstein--Uhlenbeck generator on \(L^2(\gamma_d)\), and that the
number operator is
\[
N:=-L.
\]
Thus
\[
N\mathrm{He}_\alpha=|\alpha|\mathrm{He}_\alpha.
\]
Equivalently, if
\[
f=\sum_{m=0}^{\infty}\Pi_m f
\]
is the Hermite chaos decomposition of \(f\in L^2(\gamma_d)\), then
\[
Nf=\sum_{m=0}^{\infty}m\,\Pi_m f
\]
whenever the right-hand side belongs to \(L^2(\gamma_d)\).

For \(s>0\), define the fractional Hermite-Sobolev domain
\[
\mathcal D(N^{s/2})
:=
\left\{
f\in L^2(\gamma_d):
\sum_{m=0}^{\infty}
m^s\|\Pi_m f\|_{L^2(\gamma_d)}^2<\infty
\right\}.
\]
For \(f\in\mathcal D(N^{s/2})\), one has
\[
\|N^{s/2}f\|_{L^2(\gamma_d)}^2
=
\sum_{m=0}^{\infty}
m^s\|\Pi_m f\|_{L^2(\gamma_d)}^2.
\]

The standard Hermite-Sobolev estimate follows immediately from the spectral
decomposition. Indeed,
\[
\|f-\Pi_{\le M}f\|_{L^2(\gamma_d)}^2
=
\sum_{m=M+1}^{\infty}
\|\Pi_mf\|_{L^2(\gamma_d)}^2.
\]
If \(f\in\mathcal D(N^{s/2})\), then
\[
\begin{aligned}
\|f-\Pi_{\le M}f\|_{L^2(\gamma_d)}^2
&=
\sum_{m=M+1}^{\infty}
\|\Pi_mf\|_{L^2(\gamma_d)}^2 \\
&\le
(M+1)^{-s}
\sum_{m=M+1}^{\infty}
m^s\|\Pi_mf\|_{L^2(\gamma_d)}^2 \\
&\le
(M+1)^{-s}
\|N^{s/2}f\|_{L^2(\gamma_d)}^2.
\end{aligned}
\]
Therefore
\[
\|f-\Pi_{\le M}f\|_{L^2(\gamma_d)}
\le
(M+1)^{-s/2}
\|N^{s/2}f\|_{L^2(\gamma_d)}.
\]
This estimate gives only algebraic decay in \(M\).

A sharper general estimate is available when the Hermite coefficients of
\(f\) have exponential decay. This is the standard analytic-vector viewpoint
for a nonnegative self-adjoint number operator; see, for example,
\cite{Nelson1959,Janson1997}. For \(a>0\), define
\[
e^{aN}f
:=
\sum_{m=0}^{\infty}e^{am}\Pi_m f
\]
on the domain
\[
\mathcal D(e^{aN})
:=
\left\{
f\in L^2(\gamma_d):
\sum_{m=0}^{\infty}
e^{2am}\|\Pi_m f\|_{L^2(\gamma_d)}^2<\infty
\right\}.
\]
If \(f\in\mathcal D(e^{aN})\), then
\[
\|e^{aN}f\|_{L^2(\gamma_d)}^2
=
\sum_{m=0}^{\infty}
e^{2am}\|\Pi_m f\|_{L^2(\gamma_d)}^2.
\]
Hence
\[
\begin{aligned}
\|f-\Pi_{\le M}f\|_{L^2(\gamma_d)}^2
&=
\sum_{m=M+1}^{\infty}
\|\Pi_mf\|_{L^2(\gamma_d)}^2 \\
&\le
e^{-2a(M+1)}
\sum_{m=M+1}^{\infty}
e^{2am}\|\Pi_mf\|_{L^2(\gamma_d)}^2 \\
&\le
e^{-2a(M+1)}
\|e^{aN}f\|_{L^2(\gamma_d)}^2.
\end{aligned}
\]
Therefore
\[
\|f-\Pi_{\le M}f\|_{L^2(\gamma_d)}
\le
e^{-a(M+1)}
\|e^{aN}f\|_{L^2(\gamma_d)}.
\]
This is the standard analytic-vector Hermite estimate.

We now apply this estimate to the Gaussian density ratio \(K_{u,S}\). By the
Hermite coefficient formula,
\[
K_{u,S}
=
\sum_{\alpha\in\mathbb N^d}
\frac{1}{\alpha!}A_\alpha(S,u)\mathrm{He}_\alpha
\qquad\text{in }L^2(\gamma_d).
\]
Since
\[
N\mathrm{He}_\alpha=|\alpha|\mathrm{He}_\alpha,
\]
we formally obtain
\[
e^{aN}K_{u,S}
=
\sum_{\alpha\in\mathbb N^d}
\frac{e^{a|\alpha|}}{\alpha!}
A_\alpha(S,u)\mathrm{He}_\alpha.
\]
The weighted homogeneity of the coefficients gives
\[
A_\alpha(e^{2a}S,e^a u)
=
e^{a|\alpha|}A_\alpha(S,u).
\]
Therefore, whenever the right-hand side belongs to \(L^2(\gamma_d)\),
\[
e^{aN}K_{u,S}
=
K_{e^a u,e^{2a}S}
\qquad\text{in }L^2(\gamma_d).
\]

By the exact \(L^2\) criterion for Gaussian density ratios,
\[
K_{e^a u,e^{2a}S}\in L^2(\gamma_d)
\]
if and only if
\[
\|e^{2a}S\|_{\mathrm{op}}<1.
\]
Equivalently,
\[
e^{2a}\|S\|_{\mathrm{op}}<1.
\]
Thus the admissible values of \(a\) are precisely
\[
0<a<\frac12\log\frac1{\|S\|_{\mathrm{op}}}.
\]
The endpoint
\[
a_*:=\frac12\log\frac1{\|S\|_{\mathrm{op}}}
\]
is not admissible, since it corresponds to the \(L^2\)-critical boundary
\[
\|e^{2a_*}S\|_{\mathrm{op}}=1.
\]

Consequently, for every
\[
0<a<a_*,
\]
the analytic-vector estimate gives
\[
\|K_{u,S}-\Pi_{\le M}K_{u,S}\|_{L^2(\gamma_d)}
\le
e^{-a(M+1)}
\|K_{e^a u,e^{2a}S}\|_{L^2(\gamma_d)}.
\]
Equivalently, if
\[
r>\|S\|_{\mathrm{op}}^{1/2},
\]
then one may choose \(a>0\) such that
\[
e^{-a}<r
\qquad\text{and}\qquad
e^{2a}\|S\|_{\mathrm{op}}<1.
\]
Hence there exists a constant \(C=C(u,S,r)>0\) such that
\[
\|K_{u,S}-\Pi_{\le M}K_{u,S}\|_{L^2(\gamma_d)}
\le
C r^M,
\qquad M\ge0.
\]

Thus the exponential upper bound obtained from the standard analytic-vector
estimate has the same root rate as the bound obtained from the exact
block-energy generating function. The additional content of the present work
is that, for Gaussian density ratios, the maximal admissible analytic radius is
computed explicitly, the corresponding root rate
\[
\|S\|_{\mathrm{op}}^{1/2}
\]
is shown to be sharp, and the Hermite block energies admit the exact generating
function
\[
\sum_{m=0}^{\infty}
t^{2m}
\|\Pi_mK_{u,S}\|_{L^2(\gamma_d)}^2
=
\det(I-t^4S^2)^{-1/2}
\exp\!\left(
t^2u^T(I-t^2S)^{-1}u
\right).
\]

\section{Precise isotropic asymptotics}

In this section we specialize the covariance defect \(S\) to the isotropic
heating case
\[
S=\sigma^2 I,
\qquad 0<|\sigma|<1.
\]
Thus
\[
K_{u,\sigma}=K_{u,\sigma^2 I}.
\]
This parametrization covers the positive scalar covariance defect \(S=sI\)
with \(s=\sigma^2>0\).  The general scalar defect \(S=sI\), \(-1<s<1\),
has sharp Hermite root rate \(|s|^{1/2}\) by the full-covariance theorem
above.  The precise coefficient and tail asymptotics below are stated for the
heating case \(s>0\), where \(s=\sigma^2\).

\subsection{Analytic radius under isotropic scaling}

Under the scaling parameter \(\varepsilon\), this gives
\[
K_{\varepsilon u,\varepsilon\sigma}
=
K_{\varepsilon u,\varepsilon^2\sigma^2 I},
\]
which is the isotropic counterpart of the anisotropic scaling
\(K_{tu,t^2S}\).

\begin{corollary}[Optimal analytic radius in the isotropic scaling parameter]
\label{cor:optimal_analytic_radius_scaling_parameter}
Fix \(u,v\in\mathbb R^d\) and \(0\neq\sigma\in\mathbb R\). Consider the
complexified scaling parameter \(\varepsilon\in\mathbb C\), and define
\(K_{\varepsilon u,\varepsilon\sigma}(v)\) by the explicit density-ratio
formula. Then the map
\[
\varepsilon\longmapsto K_{\varepsilon u,\varepsilon\sigma}(v)
\]
is holomorphic on the disk
\[
|\varepsilon|<\frac{1}{|\sigma|}.
\]
This radius is optimal. More precisely, the explicit formula has
non-removable singularities at
\[
\varepsilon=\pm\frac{i}{\sigma},
\]
and therefore the Taylor series at \(\varepsilon=0\) has radius of convergence
exactly
\[
\frac1{|\sigma|}.
\]
\end{corollary}

\begin{proof}
By the explicit formula,
\[
K_{\varepsilon u,\varepsilon\sigma}(v)
=
(1+\varepsilon^2\sigma^2)^{-d/2}
\exp\!\left(
\frac{\varepsilon\,u\cdot v}{1+\varepsilon^2\sigma^2}
+
\frac{\varepsilon^2\sigma^2}{2(1+\varepsilon^2\sigma^2)}|v|^2
-
\frac{\varepsilon^2|u|^2}{2(1+\varepsilon^2\sigma^2)}
\right).
\]
On the disk centered at \(0\) and not containing any zero of
\(1+\varepsilon^2\sigma^2\), we choose the branch of
\[
(1+\varepsilon^2\sigma^2)^{-d/2}
\]
which is normalized to be \(1\) at \(\varepsilon=0\). With this branch, the
above expression is holomorphic as long as
\[
1+\varepsilon^2\sigma^2\neq0.
\]
The nearest zeros of this denominator are
\[
\varepsilon=\pm\frac{i}{\sigma},
\]
and both have modulus \(1/|\sigma|\).

These boundary points are genuine singularities. To see this, let
\[
\varepsilon_0\in\left\{\frac{i}{\sigma},-\frac{i}{\sigma}\right\}.
\]
Near \(\varepsilon_0\), the factor
\[
1+\varepsilon^2\sigma^2
\]
has a simple zero. Hence the prefactor
\[
(1+\varepsilon^2\sigma^2)^{-d/2}
\]
has no holomorphic extension across \(\varepsilon_0\): it is a pole when
\(d\) is even and a branch point when \(d\) is odd.

The exponential factor cannot remove this obstruction. If the rational
function in the exponent has a pole at \(\varepsilon_0\), then the exponential
factor has an essential singularity there. If that pole is removable for a
special choice of \(u\) and \(v\), then the exponential factor extends
holomorphically and is nonzero at \(\varepsilon_0\), so the singularity of the
prefactor remains. Thus \(K_{\varepsilon u,\varepsilon\sigma}(v)\) has a
non-removable singularity at each of
\[
\varepsilon=\pm\frac{i}{\sigma}.
\]
Therefore the Taylor series at \(\varepsilon=0\) has radius of convergence
exactly \(1/|\sigma|\).
\end{proof}

\subsection{Block-energy coefficient asymptotics}

Assume throughout this subsection that
\[
0<|\sigma|<1.
\]
Set
\[
s:=\sigma^2.
\]
For
\[
a_m(u,\sigma):=
\|\Pi_mK_{u,\sigma}\|_{L^2(\gamma_d)}^2,
\]
the isotropic block-energy generating function is
\[
F(z)
:=
\sum_{m=0}^\infty a_m(u,\sigma)z^m
=
(1-s^2z^2)^{-d/2}
\exp\!\left(
\frac{z|u|^2}{1-sz}
\right).
\]

We first record the coefficient asymptotic needed for the nonzero mean case.
Here and below, \([z^m]G(z)\) denotes the coefficient of \(z^m\) in the
Taylor expansion of \(G\) at \(z=0\).
\begin{lemma}[Coefficient asymptotics at an exponential singularity]
\label{lem:exponential_singularity_coefficients}
Let \(\beta>0\) and \(\lambda>0\). Then, as \(m\to\infty\),
\begin{equation}
[z^m]\,
(1-z)^{-\beta}
\exp\!\left(
\frac{\lambda z}{1-z}
\right)
\sim
\frac{1}{2\sqrt\pi}\,
\lambda^{\frac14-\frac\beta2}\,
m^{\frac\beta2-\frac34}
\exp\!\left(
2\sqrt{\lambda m}-\frac{\lambda}{2}
\right).
\label{eq:basic_exponential_singularity_asymptotic}
\end{equation}
\end{lemma}

\begin{proof}
By the generating function for the generalized Laguerre polynomials,
\[
\sum_{m=0}^\infty L_m^{(\alpha)}(x)z^m
=
(1-z)^{-\alpha-1}
\exp\!\left(
-\frac{xz}{1-z}
\right),
\]
we have, with
\[
\alpha=\beta-1,
\qquad
x=-\lambda,
\]
that
\[
[z^m]\,
(1-z)^{-\beta}
\exp\!\left(
\frac{\lambda z}{1-z}
\right)
=
L_m^{(\beta-1)}(-\lambda).
\]
The Plancherel--Rotach asymptotic formula for generalized Laguerre
polynomials at a fixed negative argument \cite{Szego1975,DLMF} gives, for
\(\lambda>0\),
\[
L_m^{(\alpha)}(-\lambda)
\sim
\frac{1}{2\sqrt\pi}\,
\lambda^{-\frac{\alpha}{2}-\frac14}
m^{\frac{\alpha}{2}-\frac14}
\exp\!\left(
2\sqrt{\lambda m}-\frac{\lambda}{2}
\right).
\]
Substituting \(\alpha=\beta-1\), we obtain
\[
-\frac{\alpha}{2}-\frac14
=
\frac14-\frac\beta2,
\qquad
\frac{\alpha}{2}-\frac14
=
\frac\beta2-\frac34.
\]
Therefore
\[
[z^m]\,
(1-z)^{-\beta}
\exp\!\left(
\frac{\lambda z}{1-z}
\right)
\sim
\frac{1}{2\sqrt\pi}\,
\lambda^{\frac14-\frac\beta2}
m^{\frac\beta2-\frac34}
\exp\!\left(
2\sqrt{\lambda m}-\frac{\lambda}{2}
\right).
\]
This proves the lemma.
\end{proof}

\begin{theorem}[Precise block-energy asymptotics, isotropic nonzero mean case]
\label{thm:isotropic_precise_asymptotics}
Assume \(0<|\sigma|<1\) and \(u\neq0\). Then
\begin{equation}
a_m(u,\sigma)
\sim
C_{d,u,\sigma}\,
m^{\frac d4-\frac34}
\exp\!\left(
\frac{2|u|}{|\sigma|}\sqrt m
\right)
|\sigma|^{2m},
\qquad m\to\infty,
\label{eq:isotropic_precise_block_energy_asymptotic}
\end{equation}
where
\begin{equation}
C_{d,u,\sigma}
=
\frac{2^{-d/2}}{2\sqrt\pi}
\exp\!\left(
-\frac{|u|^2}{2\sigma^2}
\right)
\left(
\frac{|u|^2}{\sigma^2}
\right)^{\frac14-\frac d4}.
\label{eq:isotropic_precise_constant_corrected}
\end{equation}
Consequently,
\begin{equation}
\|\Pi_mK_{u,\sigma}\|_{L^2(\gamma_d)}
\sim
C_{d,u,\sigma}^{1/2}
m^{\frac d8-\frac38}
\exp\!\left(
\frac{|u|}{|\sigma|}\sqrt m
\right)
|\sigma|^m.
\label{eq:isotropic_precise_block_norm_asymptotic}
\end{equation}
\end{theorem}

\begin{proof}
Set
\[
s=\sigma^2,
\qquad
\lambda=\frac{|u|^2}{s}
=
\frac{|u|^2}{\sigma^2}.
\]
Since \(u\neq0\), we have \(\lambda>0\). The block-energy generating function is
\[
F(z)
=
(1-s^2z^2)^{-d/2}
\exp\!\left(
\frac{z|u|^2}{1-sz}
\right).
\]
Introduce the rescaled variable
\[
y=sz.
\]
Then
\[
F(z)=\Phi(y),
\qquad y=sz,
\]
where
\[
\Phi(y)
=
(1-y^2)^{-d/2}
\exp\!\left(
\lambda\frac{y}{1-y}
\right).
\]
Therefore, if
\[
\Phi(y)=\sum_{m=0}^{\infty}b_m y^m,
\]
then
\[
F(z)=\sum_{m=0}^{\infty}b_m s^m z^m,
\]
and hence
\[
[z^m]F(z)=s^m[y^m]\Phi(y).
\]

The singularities of \(\Phi\) on the circle \(|y|=1\) occur at \(y=1\) and
\(y=-1\). The point \(y=1\) carries the exponential singularity
\[
\exp\!\left(
\lambda\frac{y}{1-y}
\right),
\]
whereas at \(y=-1\) the exponential factor is analytic and finite:
\[
\exp\!\left(
\lambda\frac{-1}{1-(-1)}
\right)
=
\exp\!\left(-\frac{\lambda}{2}\right).
\]
Thus \(y=-1\) contributes only an algebraic singular contribution coming from
\((1-y^2)^{-d/2}\). Indeed, near \(y=-1\),
\[
(1-y^2)^{-d/2}
=
(1-y)^{-d/2}(1+y)^{-d/2},
\]
where
\[
(1-y)^{-d/2}
\exp\!\left(\lambda\frac{y}{1-y}\right)
\]
is analytic and nonzero at \(y=-1\). By the standard transfer theorem for
algebraic singularities, the contribution from \(y=-1\) is, up to the
oscillatory factor \((-1)^m\), of algebraic order
\[
O\!\left(m^{\frac d2-1}\right).
\]
By contrast, the singularity at \(y=1\) gives the subexponential factor
\[
\exp\!\left(2\sqrt{\lambda m}\right),
\]
as in Lemma~\ref{lem:exponential_singularity_coefficients}. Since
\[
m^{\frac d2-1}
=
o\!\left(
m^{\frac d4-\frac34}
\exp(2\sqrt{\lambda m})
\right),
\]
the contribution from \(y=-1\) is asymptotically negligible. Hence the
dominant coefficient asymptotics are determined by the singularity at \(y=1\).

Near \(y=1\),
\[
(1-y^2)^{-d/2}
=
(1-y)^{-d/2}(1+y)^{-d/2}.
\]
The factor
\[
(1+y)^{-d/2}
\]
is analytic and nonzero at \(y=1\), with value
\[
(1+1)^{-d/2}=2^{-d/2}.
\]
Therefore the dominant singular expansion at \(y=1\) is
\[
\Phi(y)
\sim
2^{-d/2}
(1-y)^{-d/2}
\exp\!\left(
\lambda\frac{y}{1-y}
\right).
\]
By Lemma~\ref{lem:exponential_singularity_coefficients} with
\[
\beta=\frac d2,
\]
we obtain
\[
[y^m]\Phi(y)
\sim
2^{-d/2}
\frac{1}{2\sqrt\pi}
\lambda^{\frac14-\frac d4}
m^{\frac d4-\frac34}
\exp\!\left(
2\sqrt{\lambda m}-\frac{\lambda}{2}
\right).
\]
Since
\[
[z^m]F(z)=s^m[y^m]\Phi(y),
\]
and
\[
s^m=\sigma^{2m}=|\sigma|^{2m},
\qquad
\sqrt{\lambda}
=
\frac{|u|}{|\sigma|},
\]
we obtain
\[
a_m(u,\sigma)
\sim
\frac{2^{-d/2}}{2\sqrt\pi}
\lambda^{\frac14-\frac d4}
m^{\frac d4-\frac34}
\exp\!\left(
2\frac{|u|}{|\sigma|}\sqrt m
-\frac{|u|^2}{2\sigma^2}
\right)
|\sigma|^{2m}.
\]
This is exactly
\eqref{eq:isotropic_precise_block_energy_asymptotic} with
\[
C_{d,u,\sigma}
=
\frac{2^{-d/2}}{2\sqrt\pi}
\exp\!\left(
-\frac{|u|^2}{2\sigma^2}
\right)
\left(
\frac{|u|^2}{\sigma^2}
\right)^{\frac14-\frac d4}.
\]
Taking square roots gives
\eqref{eq:isotropic_precise_block_norm_asymptotic}.
\end{proof}

\begin{remark}[Zero mean case]
If \(u=0\), then the exponential singularity disappears. In this case
\[
F(z)
=
(1-\sigma^4 z^2)^{-d/2}.
\]
Hence only even chaos levels occur:
\[
a_{2n+1}(0,\sigma)=0,
\]
and
\[
a_{2n}(0,\sigma)
=
\sigma^{4n}
\frac{(d/2)_n}{n!}.
\]
Consequently, as \(n\to\infty\),
\[
a_{2n}(0,\sigma)
\sim
\frac{1}{\Gamma(d/2)}
n^{\frac d2-1}
\sigma^{4n}.
\]
Equivalently,
\[
\|\Pi_{2n}K_{0,\sigma}\|_{L^2(\gamma_d)}
\sim
\frac{1}{\Gamma(d/2)^{1/2}}
n^{\frac d4-\frac12}
|\sigma|^{2n},
\qquad
\|\Pi_{2n+1}K_{0,\sigma}\|_{L^2(\gamma_d)}=0.
\]
\end{remark}

\section{Precise isotropic tail asymptotics}

\begin{corollary}[Precise isotropic spectral tail]
\label{cor:isotropic_precise_tail}
Assume \(0<|\sigma|<1\) and \(u\neq0\). Then
\begin{equation}
\|K_{u,\sigma}-\Pi_{\le M}K_{u,\sigma}\|_{L^2(\gamma_d)}^2
\sim
\frac{1}{1-\sigma^2}
a_{M+1}(u,\sigma).
\label{eq:isotropic_precise_tail_squared}
\end{equation}
Consequently,
\begin{equation}
\|K_{u,\sigma}-\Pi_{\le M}K_{u,\sigma}\|_{L^2(\gamma_d)}
\sim
\left(
\frac{C_{d,u,\sigma}}{1-\sigma^2}
\right)^{1/2}
(M+1)^{\frac d8-\frac38}
\exp\!\left(
\frac{|u|}{|\sigma|}\sqrt{M+1}
\right)
|\sigma|^{M+1}.
\label{eq:isotropic_precise_tail_corrected}
\end{equation}
Equivalently, one may replace \(M+1\) by \(M\) in the slowly varying
factors, at the cost of multiplying the leading constant by \(|\sigma|\).
\end{corollary}

\begin{proof}
By Theorem~\ref{thm:isotropic_precise_asymptotics},
\[
a_m(u,\sigma)
\sim
C_{d,u,\sigma}
m^{\frac d4-\frac34}
\exp\!\left(
\frac{2|u|}{|\sigma|}\sqrt m
\right)
\sigma^{2m}.
\]
Therefore
\[
\frac{a_{m+1}(u,\sigma)}{a_m(u,\sigma)}
\to
\sigma^2.
\]
Let
\[
q:=\sigma^2.
\]
Since \(0<q<1\), the ratio limit implies the standard geometric-tail
asymptotic
\[
\sum_{m=M+1}^\infty a_m(u,\sigma)
\sim
\frac{a_{M+1}(u,\sigma)}{1-q}.
\]
For completeness, we recall the argument. Fix \(\varepsilon>0\) such that
\(q+\varepsilon<1\). For all sufficiently large \(m\),
\[
q-\varepsilon
\le
\frac{a_{m+1}}{a_m}
\le
q+\varepsilon.
\]
Hence, for all sufficiently large \(M\),
\[
a_{M+1}
\le
\sum_{m=M+1}^\infty a_m
\le
a_{M+1}
\sum_{j=0}^\infty (q+\varepsilon)^j
=
\frac{a_{M+1}}{1-q-\varepsilon}.
\]
Similarly, for every fixed \(J\ge0\),
\[
\sum_{m=M+1}^{M+1+J}a_m
\ge
a_{M+1}
\sum_{j=0}^J(q-\varepsilon)^j.
\]
Letting first \(M\to\infty\), then \(J\to\infty\), and finally
\(\varepsilon\downarrow0\), we obtain
\[
\sum_{m=M+1}^\infty a_m
\sim
\frac{a_{M+1}}{1-q}.
\]
Since
\[
\|K_{u,\sigma}-\Pi_{\le M}K_{u,\sigma}\|_{L^2(\gamma_d)}^2
=
\sum_{m=M+1}^\infty a_m(u,\sigma),
\]
we get
\[
\|K_{u,\sigma}-\Pi_{\le M}K_{u,\sigma}\|_{L^2(\gamma_d)}^2
\sim
\frac{1}{1-\sigma^2}
a_{M+1}(u,\sigma).
\]
Taking square roots and substituting the asymptotic formula for
\(a_{M+1}(u,\sigma)\) proves
\eqref{eq:isotropic_precise_tail_corrected}.
\end{proof}

\section{Application to kinetic Hermite approximation}

We now interpret the preceding estimates as approximation estimates for
near-Gaussian kinetic distributions, following the general tradition of
Hermite and moment approximations in kinetic theory
\cite{Grad1949Kinetic,Levermore1996,Struchtrup2005,Villani2002}. Let \(f\ge0\) be a kinetic density and
write
\[
g(v):=\frac{f(v)}{w(v)}.
\]
In applications, the Gaussian core may be chosen in several ways.  A natural
choice is to match the mass, mean velocity, and covariance of the measure
\[
g\,d\gamma_d=f\,dv,
\]
whenever these moments are finite and the resulting covariance matrix is
positive definite.  Equivalently, in the moment-matching choice one takes
\[
\rho
=
\int_{\mathbb R^d}g\,d\gamma_d
=
\int_{\mathbb R^d}f(v)\,dv,
\]
\[
u
=
\rho^{-1}
\int_{\mathbb R^d}v\,g(v)\,d\gamma_d(v)
=
\rho^{-1}
\int_{\mathbb R^d}v\,f(v)\,dv,
\]
and
\[
I+S
=
\rho^{-1}
\int_{\mathbb R^d}
(v-u)(v-u)^Tg(v)\,d\gamma_d(v)
=
\rho^{-1}
\int_{\mathbb R^d}
(v-u)(v-u)^Tf(v)\,dv.
\]
With this choice, the residual \(h\) has vanishing mass, first moment, and
covariance defect relative to the chosen Gaussian core.  Alternatively,
\(K_{u,S}\) may be chosen as a local Gaussian approximation or as the minimizer
of a prescribed quadratic or entropic criterion.  The estimates below are
independent of this choice; they quantify the Hermite truncation error once
such a Gaussian core has been selected.  Thus the results in this section
should be read as conditional estimates: they separate the explicitly
computable Gaussian-core contribution from the residual spectral tail.  The
fixed-reference \(L^2(\gamma_d)\) Gaussian-core estimates apply to the
Gaussian part when
\[
\|S\|_{\mathrm{op}}<1.
\]
Assume that \(g\in L^2(\gamma_d)\) admits the decomposition
\begin{equation}
g
=
\rho K_{u,S}+h,
\qquad
\rho>0,
\label{eq:kinetic_gaussian_core_decomposition}
\end{equation}
where \(K_{u,S}\) is a Gaussian density ratio and
\(h\in L^2(\gamma_d)\) is the non-Gaussian remainder.

Let
\[
\Pi_{\le M}:=\sum_{m=0}^{M}\Pi_m.
\]
Then
\[
g-\Pi_{\le M}g
=
\rho\bigl(K_{u,S}-\Pi_{\le M}K_{u,S}\bigr)
+
(I-\Pi_{\le M})h.
\]
Thus the Hermite truncation error splits into a Gaussian-core contribution and
a non-Gaussian residual contribution.

\begin{theorem}[Gaussian-core Hermite approximation]
\label{thm:kinetic_gaussian_core_approximation}
Assume \(S=S^T\), \(I+S>0\), and
\[
\|S\|_{\mathrm{op}}<1.
\]
Let \(r\) satisfy
\[
\|S\|_{\mathrm{op}}^{1/2}<r<1.
\]
Then there exists \(C=C(u,S,r)>0\) such that, for every \(M\ge0\),
\begin{equation}
\|g-\Pi_{\le M}g\|_{L^2(\gamma_d)}
\le
\rho C r^M
+
\|(I-\Pi_{\le M})h\|_{L^2(\gamma_d)}.
\label{eq:kinetic_gaussian_core_error}
\end{equation}
If, in addition, \(h\in\mathcal D(N^\beta)\) for some \(\beta>0\), then
\begin{equation}
\|g-\Pi_{\le M}g\|_{L^2(\gamma_d)}
\le
\rho C r^M
+
(M+1)^{-\beta}
\|N^\beta h\|_{L^2(\gamma_d)}.
\label{eq:kinetic_gaussian_core_error_sobolev}
\end{equation}
\end{theorem}

\begin{proof}
By the triangle inequality,
\[
\|g-\Pi_{\le M}g\|_{L^2(\gamma_d)}
\le
\rho
\|K_{u,S}-\Pi_{\le M}K_{u,S}\|_{L^2(\gamma_d)}
+
\|(I-\Pi_{\le M})h\|_{L^2(\gamma_d)}.
\]
The Gaussian-core term is bounded by Theorem~\ref{thm:anisotropic_sharp_root_rate}
when \(0<\|S\|_{\mathrm{op}}<1\).

If \(S=0\), then \(K_{u,0}\) is an analytic vector for \(N\) of arbitrary
positive radius. Indeed,
\[
e^{aN}K_{u,0}=K_{e^a u,0}\in L^2(\gamma_d)
\]
for every \(a>0\). Hence, for every \(0<r<1\), choosing \(a=-\log r\) in the
analytic-vector estimate gives
\[
\|K_{u,0}-\Pi_{\le M}K_{u,0}\|_{L^2(\gamma_d)}
\le
C(u,r)r^M.
\]
Thus, in all cases \(\|S\|_{\mathrm{op}}<1\), the Gaussian-core term is bounded
by \(Cr^M\). This proves \eqref{eq:kinetic_gaussian_core_error}.

If \(h\in\mathcal D(N^\beta)\), then
\[
\begin{aligned}
\|(I-\Pi_{\le M})h\|_{L^2(\gamma_d)}^2
&=
\sum_{m=M+1}^{\infty}
\|\Pi_mh\|_{L^2(\gamma_d)}^2 \\
&\le
(M+1)^{-2\beta}
\sum_{m=M+1}^{\infty}
m^{2\beta}\|\Pi_mh\|_{L^2(\gamma_d)}^2 \\
&\le
(M+1)^{-2\beta}
\|N^\beta h\|_{L^2(\gamma_d)}^2.
\end{aligned}
\]
Taking square roots gives
\[
\|(I-\Pi_{\le M})h\|_{L^2(\gamma_d)}
\le
(M+1)^{-\beta}\|N^\beta h\|_{L^2(\gamma_d)}.
\]
Substituting this into \eqref{eq:kinetic_gaussian_core_error} proves
\eqref{eq:kinetic_gaussian_core_error_sobolev}.
\end{proof}

The preceding theorem is an a priori upper bound. It shows that the Gaussian
core contributes an exponentially small Hermite tail, whereas the residual
\(h\) contributes whatever spectral tail remains after subtracting the
Gaussian core. The sharp rate of the full approximation is obtained when the
residual tail is spectrally smaller than the Gaussian-core tail.

\begin{theorem}[Gaussian-core dominance of the truncation rate]
\label{thm:kinetic_gaussian_core_dominance}
Assume
\[
g=\rho K_{u,S}+h,
\qquad
\rho>0,
\]
with \(S=S^T\), \(I+S>0\), and
\[
0<\|S\|_{\mathrm{op}}<1.
\]
Suppose that the residual satisfies
\begin{equation}
\limsup_{M\to\infty}
\|(I-\Pi_{\le M})h\|_{L^2(\gamma_d)}^{1/M}
<
\|S\|_{\mathrm{op}}^{1/2}.
\label{eq:residual_faster_than_gaussian_core}
\end{equation}
Then the full truncation error has the same sharp root rate as the Gaussian
core:
\begin{equation}
\limsup_{M\to\infty}
\|g-\Pi_{\le M}g\|_{L^2(\gamma_d)}^{1/M}
=
\|S\|_{\mathrm{op}}^{1/2}.
\label{eq:kinetic_full_error_sharp_root_rate}
\end{equation}
\end{theorem}

\begin{proof}
Set
\[
G_M:=K_{u,S}-\Pi_{\le M}K_{u,S},
\qquad
H_M:=(I-\Pi_{\le M})h.
\]
Then
\[
g-\Pi_{\le M}g=\rho G_M+H_M.
\]
By Theorem~\ref{thm:anisotropic_sharp_root_rate},
\[
\limsup_{M\to\infty}
\|G_M\|_{L^2(\gamma_d)}^{1/M}
=
\|S\|_{\mathrm{op}}^{1/2}.
\]
The upper bound in \eqref{eq:kinetic_full_error_sharp_root_rate} follows from
the triangle inequality and the assumption
\eqref{eq:residual_faster_than_gaussian_core}.

For the lower bound, use the reverse triangle inequality:
\[
\|\rho G_M+H_M\|_{L^2(\gamma_d)}
\ge
\rho\|G_M\|_{L^2(\gamma_d)}
-
\|H_M\|_{L^2(\gamma_d)}.
\]
Since \(H_M\) has strictly smaller root rate than \(G_M\), it is negligible
along a subsequence realizing the limsup of \(\|G_M\|^{1/M}\). Therefore
\[
\limsup_{M\to\infty}
\|\rho G_M+H_M\|_{L^2(\gamma_d)}^{1/M}
\ge
\|S\|_{\mathrm{op}}^{1/2}.
\]
The matching upper and lower bounds prove
\eqref{eq:kinetic_full_error_sharp_root_rate}.
\end{proof}

In particular, if the non-Gaussian residual has a faster Hermite spectral
decay than the Gaussian core, then the covariance defect of the Gaussian core
alone determines the asymptotic truncation rate. Thus
\[
\|S\|_{\mathrm{op}}^{1/2}
\]
is not merely an upper-bound rate for the Gaussian approximation component; it
is the actual sharp geometric root rate of the full kinetic approximation
whenever the residual is spectrally negligible.

In the isotropic heating case, the previous precise asymptotics give a sharper
result.

\begin{corollary}[Precise isotropic kinetic tail]
\label{cor:kinetic_precise_isotropic_tail}
Assume
\[
g=\rho K_{u,\sigma}+h,
\qquad
\rho>0,
\qquad
0<|\sigma|<1,
\qquad
u\neq0.
\]
Suppose that
\begin{equation}
\|(I-\Pi_{\le M})h\|_{L^2(\gamma_d)}
=
o\!\left(
(M+1)^{\frac d8-\frac38}
\exp\!\left(
\frac{|u|}{|\sigma|}\sqrt{M+1}
\right)
|\sigma|^{M+1}
\right).
\label{eq:isotropic_residual_negligible_condition}
\end{equation}
Then
\begin{equation}
\|g-\Pi_{\le M}g\|_{L^2(\gamma_d)}
\sim
\rho
\left(
\frac{C_{d,u,\sigma}}{1-\sigma^2}
\right)^{1/2}
(M+1)^{\frac d8-\frac38}
\exp\!\left(
\frac{|u|}{|\sigma|}\sqrt{M+1}
\right)
|\sigma|^{M+1}.
\label{eq:kinetic_precise_isotropic_tail}
\end{equation}
\end{corollary}

\begin{proof}
By Corollary~\ref{cor:isotropic_precise_tail},
\[
\|K_{u,\sigma}-\Pi_{\le M}K_{u,\sigma}\|_{L^2(\gamma_d)}
\sim
\left(
\frac{C_{d,u,\sigma}}{1-\sigma^2}
\right)^{1/2}
(M+1)^{\frac d8-\frac38}
\exp\!\left(
\frac{|u|}{|\sigma|}\sqrt{M+1}
\right)
|\sigma|^{M+1}.
\]
The assumption \eqref{eq:isotropic_residual_negligible_condition} says exactly
that the residual tail is smaller than this Gaussian-core tail. Hence
\[
g-\Pi_{\le M}g
=
\rho\bigl(K_{u,\sigma}-\Pi_{\le M}K_{u,\sigma}\bigr)
+
(I-\Pi_{\le M})h
\]
has the same norm asymptotic as its Gaussian-core component. This proves
\eqref{eq:kinetic_precise_isotropic_tail}.
\end{proof}

\subsection{Interpretation}

The approximation error separates into
\[
\text{Gaussian-core error}
+
\text{non-Gaussian residual error}.
\]
The Gaussian-core error is not governed by a generic Sobolev regularity index.
It is governed exactly by the covariance defect of the Gaussian core relative
to the reference Gaussian. In the anisotropic case the sharp root rate is
\[
\|S\|_{\mathrm{op}}^{1/2}.
\]
In the isotropic heating parametrization \(S=\sigma^2 I\), this becomes
\[
|\sigma|.
\]

Thus fixed-reference Hermite approximation deteriorates precisely as the
Gaussian covariance approaches the \(L^2(\gamma_d)\)-critical boundary
\[
\|S\|_{\mathrm{op}}\uparrow1.
\]
Equivalently, in the isotropic notation, deterioration occurs as
\[
|\sigma|\uparrow1.
\]
The exact block-energy generating function identifies this deterioration at
the level of Hermite blocks, the sharp root-rate theorem identifies the
optimal exponential scale, and the precise isotropic asymptotics give the
leading-order tail when the Gaussian core dominates the residual.

Therefore, for near-Gaussian kinetic distributions, the Hermite truncation
error is controlled by two distinct mechanisms: the covariance defect of the
Gaussian core and the remaining Hermite regularity of the non-Gaussian
residual. When the residual is spectrally negligible, the covariance defect
alone determines the full asymptotic truncation rate.

The parabolic control principle of
Section~\ref{subsec:parabolic_control_hermite_tails} gives a complementary
interpretation of the residual mechanism.  If a localized Gaussian core or a
non-Gaussian residual family evolves by heat flow in an external parameter,
then its Hermite block energies and truncation tails remain controlled by
parameter-space subsolution identities, even when no usable closed coefficient
generating transform is available.  Thus the exact Gaussian formulas identify
the sharp closed-form contribution of the Gaussian core, while the heat-flow
energy identities provide a robust control mechanism for heat-evolved
non-Gaussian perturbations.

\section{Conclusion}

We have shown that the classical Hermite generating function has a concrete
Gaussian interpretation: it is the density ratio of two unit-temperature
Gaussian densities, one with mean \(u\) and the other the centered reference.
Applying the heat semigroup in the mean variable gives the normalized
Maxwellian ratio with the same mean and temperature \(1+\tau\).  Thus the heat
time in the mean parameter is identified with the temperature increment
relative to the fixed unit-temperature reference.

We also isolated an energy-level consequence of this parameter-space heat
flow.  Whenever an \(L^2(\gamma_d)\)-valued family satisfies the heat equation
in the mean parameter, its Hermite block energies and truncation tails satisfy
parabolic subsolution identities.  This part of the argument does not require
a closed formula for the individual Hermite coefficients.

For Gaussian density ratios with general covariance, we derived the Hermite
coefficient generating function and its weighted homogeneity in the mean and
covariance-defect parameters.  This leads to the Ornstein--Uhlenbeck covariance
of the Gaussian ratio family, the exact \(L^2(\gamma_d)\) norm formula, and the
exact generating function for total-degree Hermite block energies.  From this
generating function we obtained the sharp Hermite truncation root rate, equal
to the square root of the largest absolute covariance defect relative to the
reference Gaussian.

In the isotropic heating case, the block-energy generating function reduces to
a Laguerre-type generating function.  This gives precise block-energy and tail
asymptotics in the scalar setting.  Finally, for near-Gaussian kinetic
densities, the Hermite truncation error separates into a Gaussian-core part and
a residual part.  When the residual tail is smaller than the Gaussian-core
tail, the covariance defect of the Gaussian core determines the leading
asymptotic truncation rate.

\backmatter

\bmhead{Acknowledgements}
This work was partially supported by the KIST Institutional Program (26E0191).

\bmhead{Conflict of interest}
The author declares that they have no competing interests.

\bmhead{Data availability}
Data sharing is not applicable to this article as no datasets were generated or analysed during the current study.

\bigskip

\bibliography{references}

\end{document}